\documentclass[11 pt]{amsart}
\usepackage{amsfonts}
\usepackage{ifthen}
\usepackage{amsthm}
\usepackage{amsmath}
\usepackage{graphicx}
\usepackage{amscd,amssymb,amsthm}
\usepackage{graphicx}
\usepackage{epstopdf}
\usepackage{hyperref}
\usepackage{cleveref}
\usepackage{cite}
\usepackage{mathrsfs}

\newcounter{minutes}
\divide\time by 60
\newcounter{hours}
\multiply\time by 60 \addtocounter{minutes}{-\time}

\newtheorem{theorem}{Theorem}[section]


\begin{document}
	
	\title{Radii problems for normalized hyper-Bessel function}

	\author[E. Toklu]{Evr{\.I}m Toklu}
	\address{Department of Mathematics, Faculty of Education, A\u{g}r{\i} {\.I}brah{\i}m \c{C}e\c{c}en University, A\u{g}r{\i}, Turkey} 
	\email{evrimtoklu@gmail.com}
	
	\author[O. Kara]{Osman Kara}
	\address{Institute of Science and Technology, A\u{g}r{\i} {\.I}brah{\i}m \c{C}e\c{c}en University, A\u{g}r{\i}, Turkey} 
	\email{osmankara.0004@gmail.com}
	
	\def\thefootnote{}
	\footnotetext{ \texttt{File:~\jobname .tex,
			printed: \number\year-0\number\month-\number\day,
			\thehours.\ifnum\theminutes<10{0}\fi\theminutes}
	} \makeatletter\def\thefootnote{\@arabic\c@footnote}\makeatother
	
	\maketitle
	
\begin{abstract}
The main purpose of the present paper is to ascertain the radii of starlikeness and convexity associated with lemniscate of Bernoulli and the Janowski function, $(1+Az)/(1+Bz)$ for $-1\leq B<A\leq 1,$ of normalized hyper-Bessel function.
\end{abstract}
	
\section{\bf Introduction and the main results}
Let $\mathbb{D}_{r}$ be the open disk $\left\lbrace z\in\mathbb{C}:\left| z\right| <r\right\rbrace $ with the radius $r>0$ and let $\mathbb{D}=\mathbb{D}_{1}.$ By $\mathcal{A}$ we mean the class of analytic functions $f:\mathbb{D}_{r}\rightarrow \mathbb{C},$ which satisfy the usual normalization
conditions $f(0)=f'(0)-1=0.$ Denote by $\mathcal{S}$ the class of functions belonging to $\mathcal{A}$ which are of univalent in $\mathbb{D}_{r}$. A function $f\in \mathcal{A}$ is said to be starlike function if $f(\mathbb{D})$ is starlike domain with the respect to the origin. It is well known fact that various subclasses of starlike function can be unified by making use of the concept of subordination. A function $f\in \mathcal{A}$ is said to be subordinate to a function $g\in \mathcal{A},$ written as $f(z)\prec g(z),$ if there exist a Schwarz function $w$ with $w(0)=0$ such that $f(z)=g(w(z)).$ In addition, we know that if $g$ is a univalent function, then $f(z)\prec g(z)$ if and only if $f(0)=g(0)$ and $f(\mathbb{D})\subset g(\mathbb{D}).$ For an analytic function $\varphi,$ let $\mathcal{S}^{\star}(\varphi)$ denote the class of all analytic functions satisfying $1+zf'(z)/f(z)\prec \varphi(z).$ By $\mathcal{K}(\varphi)$ we mean the class of all analytic functions satisfying $1+zf''(z)/f'(z)\prec \varphi(z).$ It is worth mentioning that these classes include respectively several famous subclasses of starlike and convex functions. For instance, the class $\mathcal{S}^{\star}_{\mathcal{L}}:=\mathcal{S}^{\star}(\sqrt{1+z})$ denotes the class of lemniscate starlike functions introduced and investigated by Sok\'ol and Stankiewich \cite{SS} and the class $\mathcal{K}_{\mathcal{L}}:=\mathcal{K}(\sqrt{1+z})$ represents the class  of lemniscate convex functions. Moreover, for $-1\leq B<A\leq 1,$ the class $\mathcal{S}^{\star}[A,B]:=\mathcal{S}^{\star}((1+Az)/(1+Bz))$ is the class of Janowski starlike functions and $\mathcal{K}[A,B]:=\mathcal{K}((1+Az)/(1+Bz))$ is the class of Janowski convex functions \cite{Janowski}.

Given a class of functions $\mathcal{M}\subset\mathcal{A}$ and a function $f\in \mathcal{A},$ the $\mathcal{M}-$radius of the function $f$ is the largest number $r$ with $0\leq r \leq 1$ such that $f_{r}\in \mathcal{M},$ where $f_{r}(z):=f(rz)/r.$ If we choose $\mathcal{M}=\mathcal{S}^{\star}_{\mathcal{L}},$ the $\mathcal{M}-$radius of the function $f$, which is represented by $r^{\star}_{\mathcal{L}}(f),$ is called the radius of lemniscate starlikeness. It is indeed the largest $r$ with $0\leq r \leq 1$ such that
\[\left|\left(\frac{zf'(z)}{f(z)} \right)^2-1  \right|<1 \quad (\left| z\right|<r ).\]
If we choose $\mathcal{M}=\mathcal{K}_{\mathcal{L}},$ the $\mathcal{M}-$radius of the function $f$, which is represented by $r^{c}_{\mathcal{L}}(f),$ is called as the radius of lemniscate convexity. It is indeed the largest $r$ with $0\leq r \leq 1$ such that
\[\left|\left(1+\frac{zf''(z)}{f'(z)} \right)^2-1  \right|<1 \quad (\left| z\right|<r ).\]
If we take $\mathcal{M}=\mathcal{S}^{\star}[A,B]$ or $\mathcal{M}=\mathcal{K}[A,B],$ the respective $\mathcal{M}-$radii, which are represented by $r^{\star}_{A,B}(f)$ and $r^{c}_{A,B}(f),$ are called as the radii of Janowski starlikeness and Janowski convexity. These are respectively the largest  $r$ with $0\leq r \leq 1$ such that
\[ \left|\frac{\left( zf'(z)/f(z)\right)-1 }{A-Bzf'(z)/f(z)} \right|<1 \text{ \ and \ }      \left|\frac{zf''(z)/f'(z) }{A-B\left( 1+zf''(z)/f'(z)\right) }\right| <1 \quad (\left| z\right|<r ).\]

Recently, there has been a vivid interest on some geometric properties such as univalency,
starlikeness, convexity and uniform convexity of various special functions such as hyper-Bessel, Wright, $q-$Bessel and Mittag-Leffler functions (see \cite{AB}, \cite{ABS}, \cite{AP}, \cite{BTK}, \cite{TAO}). Fore more details on the radius problems, one may consult on \cite{AJR}, \cite{Goodman}, \cite{MKR1}, \cite{VR}. Moreover, in \cite{MKR2} the authors determined the radii of starlikeness and convexity associated with lemniscate of Bernoulli and the Janowski function $(1+Az)/(1+Bz)$.  Motivated by the above series of papers on geometric properties of special functions, in this paper our aim is to determine the radii of lemniscate starlikeness, lemniscate convexity, Janowski starlikeness and Janowski convexity of certain normalized hyper-Bessel function.

Let us consider the hyper-Bessel function defined by
\begin{equation}\label{HB1}
J_{\alpha_{d}}(z)=\frac{\left( \frac{z}{d+1}\right)^{\alpha_{1}+\alpha_{2}+\dots+\alpha_{d} }}{\Gamma(\alpha_{1}+1)\dots\Gamma(\alpha_{d}+1)}{}_{0}F_{d}\left( \begin{matrix}-&\\(\alpha_{d}+1)&\end{matrix};-\left(\frac{z}{d+1} \right)^{d+1}  \right), 
\end{equation}
where the notation
\begin{equation}\label{HB2}
{}_{p}F_{q}\left(\begin{matrix}(\beta_{p})&\\(\gamma_{q})&\end{matrix};x\right)=\sum_{n\geq0}\frac{(\beta_{1})_{n}(\beta_{2})_{n}\dots(\beta_{p})_{n}}{(\gamma_{1})_{n}(\gamma_{2})_{n}\dots(\gamma_{q})_{n}}\frac{x^n}{n!}
\end{equation}
stands for the generalized hypergeometric function, $(\beta)_{n}$ is the shifted factorial (or
Pochhammer’s symbol) defined by $(\beta)_{0}=1,$ $(\beta)_{n}=\beta(\beta+1)\dots(\beta+n-1), n\geq1$ and the contracted notation $\alpha_{d}$ is used to abbreviate the array of $d$ parameters $\alpha_{1},\dots\alpha_{d}.$

By using Eqs. \eqref{HB1} and \eqref{HB2} it is easy to deduce that the function $z\mapsto J_{\alpha_{d}}(z)$ has the following infinite sum representation:
\begin{equation}\label{HB3}
J_{\alpha_{d}}(z)=\sum_{n\geq0}\frac{(-1)^n}{n!\Gamma(\alpha_{1}+1+n)\dots\Gamma(\alpha_{d}+1+n)}\left( \frac{z}{d+1}\right)^{n(d+1)+\alpha_{1}+\dots+\alpha_{d}}.
\end{equation}
It is obvious that by choosing $d=1$ and puttinq $\alpha_{1}=\nu$ in \eqref{HB3} we obtain the classical Bessel function. The normalized hyper-Bessel function $\mathcal{J}_{\alpha_{d}}(z)$ defined by
\begin{equation}\label{HB4}
J_{\alpha_{d}}(z)=\frac{\left( \frac{z}{d+1}\right)^{\alpha_{1}+\dots+\alpha_{d}}}{\Gamma(\alpha_{1}+1)\dots \Gamma(\alpha_{d}+1)}\mathcal{J}_{\alpha_{d}}(z).
\end{equation}
By combining Eqs. \eqref{HB3} and \eqref{HB4} we obtain the following infinite sum representation:
\begin{equation}\label{HB5}
\mathcal{J}_{\alpha_{d}}(z)=\sum_{n\geq0}\frac{(-1)^n}{n!(\alpha_{1}+1)_{n}\dots(\alpha_{d}+1)_{n}}\left(\frac{z}{d+1} \right)^{n(d+1)}.
\end{equation}
Since the function $\mathcal{J}_{\alpha_{d}}$ does not belong to the class $\mathcal{A}$ we focus on the following normalized form
\begin{equation}\label{HB6}
f_{\alpha_{d}}(z)=z\mathcal{J}_{\alpha_{d}}(z)=\sum_{n\geq0}\frac{(-1)^n}{n!(d+1)^{n(d+1)}(\alpha_{1}+1)_{n}\dots(\alpha_{d}+1)_{n}}z^{n(d+1)+1}
\end{equation}
so that the function $f_{\alpha_{d}}\in\mathcal{A}.$

The Weierstrassian canonical product expansion of the function $J_{\alpha_{d}}$ reads as (see \cite[Eq.(5.5)]{CR})
\begin{equation}\label{HB7}
J_{\alpha_{d}}(z)=\frac{(\frac{z}{d+1})^{\alpha_{1}+\dots+\alpha_{d}}}{\Gamma(\alpha_{1}+1)\dots\Gamma(\alpha_{d}+1)}\prod_{n\geq1}\left( 1-\frac{z^{d+1}}{j_{\alpha_{d},n}^{d+1}}\right), 
\end{equation} 
where $j_{\alpha_{d},n}$ is the $n$th positive zero of the function $\mathcal{J}_{\alpha_{d}}.$

In light of Eqs. \eqref{HB4} and \eqref{HB7}, we get
\begin{equation}\label{HB8}
\mathcal{J}_{\alpha_{d}}(z)=\prod_{n\geq1}\left( 1-\frac{z^{d+1}}{j_{\alpha_{d},n}^{d+1}}\right)
\end{equation}
and consequently
\begin{equation}\label{HB9}
f_{\alpha_{d}}(z)=z\prod_{n\geq1}\left( 1-\frac{z^{d+1}}{j_{\alpha_{d},n}^{d+1}}\right).
\end{equation}
\subsection{Lemniscate starlikeness and lemniscate convexity of normalized hyper-Bessel function}
This section is devoted to determine the radii of lemniscate starlikeness and lemniscate convexity of the normalized hyper-Bessel function. 

\begin{theorem}
Let $\alpha_{i}>-1$ for $i\in \left\lbrace 1,2,\dots,d\right\rbrace .$ Then the radius of lemniscate starlikeness $r^{\star}_{\mathcal{L}}(f_{\alpha_{d}})$ of the normalized hyper-Bessel function $z\mapsto f_{\alpha_{d}}(z)=z\mathcal{J}_{\alpha_{d}}(z)$ is the smallest positive root of the equation
\[r^2\left(\mathcal{J}_{\alpha_{d}}'(r) \right)^2-2r\mathcal{J}_{\alpha_{d}}'(r)\mathcal{J}_{\alpha_{d}}(r)+3\mathcal{J}_{\alpha_{d}}^2(r)=0. \] 
\end{theorem}
\begin{proof}
By means of Eq. \eqref{HB8} we have
\begin{equation}\label{HB10}
\frac{z\mathcal{J}_{\alpha_{d}}'(z)}{\mathcal{J}_{\alpha_{d}}(z)}=-(d+1)\sum_{n\geq1}\frac{z^{d+1}}{j_{\alpha_{d},n}^{d+1}-z^{d+1}}.
\end{equation}
Taking into consideration the normalization \eqref{HB6}, it follows from the equation \eqref{HB10} that
\begin{equation}\label{HB11}
\frac{zf_{\alpha_{d}}'(z)}{f_{\alpha_{d}}(z)}=1+\frac{z\mathcal{J}_{\alpha_{d}}'(z)}{\mathcal{J}_{\alpha_{d}}(z)}=1-(d+1)\sum_{n\geq1}\frac{z^{d+1}}{j_{\alpha_{d},n}^{d+1}-z^{d+1}}.
\end{equation}
In light of Eq. \eqref{HB11}, we get
\begin{align}\label{HB12}
\left|\left( \frac{zf_{\alpha_{d}}'(z)}{f_{\alpha_{d}}(z)}\right)^{2}-1  \right|&\leq \left((d+1)\sum_{n\geq1}\frac{\left| z\right| ^{d+1}}{j_{\alpha_{d},n}^{d+1}-\left| z\right| ^{d+1}} \right)\left(2+ (d+1)\sum_{n\geq1}\frac{\left| z\right| ^{d+1}}{j_{\alpha_{d},n}^{d+1}-\left| z\right| ^{d+1}}\right)   \notag \\
&=\left(\frac{\left| z\right| f_{\alpha_{d}}'(\left| z\right| )}{f_{\alpha_{d}}(\left| z\right| )} \right)^2-4\left( \frac{\left| z\right| f_{\alpha_{d}}'(\left| z\right| )}{f_{\alpha_{d}}(\left| z\right| )}\right) +3.
\end{align}
Suppose that $r^{\star}$ is the smallest positive root of the equation
\[\left(\frac{r f_{\alpha_{d}}'(r )}{f_{\alpha_{d}}(r )} \right)^2-4\left( \frac{r f_{\alpha_{d}}'(r )}{f_{\alpha_{d}}(r )}\right) +2=0,\]
then the inequality 
$$\left| \left(\frac{r f_{\alpha_{d}}'(r )}{f_{\alpha_{d}}(r )}\right)^{2}-1 \right| <1$$
holds true for $\left|z \right|<r^{\star}. $
By virtue of Eq. \eqref{HB11} we deduce that the zeros of the above equation coincide with those of equation
\begin{equation}\label{HB13}
r^2\left(\mathcal{J}_{\alpha_{d}}'(r) \right)^2-2r\mathcal{J}_{\alpha_{d}}'(r)\mathcal{J}_{\alpha_{d}}(r)+3\mathcal{J}_{\alpha_{d}}^2(r)=0.
\end{equation}

In order to end the proof, we need to show that equation \eqref{HB13} has a unique root in $\left(0,j_{\alpha_{d},1}\right).$ The function $u_{\alpha_{d}}(z):\left(0,j_{\alpha_{d},1}\right)\rightarrow \mathbb{R},$ defined by
\[u_{\alpha_{d}}(z)=\left(\frac{r f_{\alpha_{d}}'(r )}{f_{\alpha_{d}}(r )} \right)^2-4\left( \frac{r f_{\alpha_{d}}'(r )}{f_{\alpha_{d}}(r )}\right) +2,\]
is continuous and strictly increasing  since
\[u_{\alpha_{d}}'(z)=2\sum_{n\geq1}\frac{(d+1)^{2}r^{d}j_{\alpha_{d},n}^{d+1}}{\left( j_{\alpha_{d},n}^{d+1}-r^{d+1}\right)^2 }\left(1+(d+1)\sum_{n\geq1}\frac{r^{d+1}}{j_{\alpha_{d},n}^{d+1}-r^{d+1}} \right)>0. \]
Observe also that
\[\lim_{r\searrow 0}u_{\alpha_{d}}(z)=-1<0 \text{ \ and \ } \lim_{r\nearrow j_{\alpha_{d},1}}u_{\alpha_{d}}(z)=\infty.\]
Due to the Intermediate Value Theorem, we conclude that the equation $u_{\alpha_{d}}(z)=0$ has a unique root in $\left(0,j_{\alpha_{d},1}\right).$ This means that the lemniscate starlike radius of the function $f_{\alpha_{d}}(z),$ say $r^{\star}_{\mathcal{L}}(f_{\alpha_{d}}),$ is the unique zero of $u_{\alpha_{d}}(z)$ in $\left(0,j_{\alpha_{d},1}\right)$ or of the equation \eqref{HB13}.
\end{proof}

\begin{theorem}
	Let $\alpha_{i}>-1$ for $i\in \left\lbrace 1,2,\dots,d\right\rbrace .$ Then the radius of lemniscate convexity $r^{c}_{\mathcal{L}}(f_{\alpha_{d}})$ of the normalized hyper-Bessel function $z\mapsto f_{\alpha_{d}}(z)=z\mathcal{J}_{\alpha_{d}}(z)$ is the smallest positive root of the equation
	\[\left(\frac{r^{2}\mathcal{J}_{\alpha_{d}}''(r)+2r\mathcal{J}_{\alpha_{d}}'(r)}{\mathcal{J}_{\alpha_{d}}(r)+r\mathcal{J}_{\alpha_{d}}'(r)}\right)^{2}-2\left(\frac{r^{2}\mathcal{J}_{\alpha_{d}}''(r)+2r\mathcal{J}_{\alpha_{d}}'(r)}{\mathcal{J}_{\alpha_{d}}(r)+r\mathcal{J}_{\alpha_{d}}'(r)} \right)-1=0. \]
\end{theorem}
\begin{proof}
From \cite[Theorem 1]{ABS} we have
\begin{equation}\label{HB14}
f_{\alpha_{d}}'(z)=\prod_{n\geq1}\left(1-\frac{z^{d+1}}{\gamma_{\alpha_{d},n}^{d+1}} \right), 
\end{equation}
where $\gamma_{\alpha_{d},n}$ is the $n$th positive real zero of the function $f_{\alpha_{d}}'(z).$ With the aid of Eq. \eqref{HB14} it is easily seen that 
\begin{equation}\label{HB15}
1+\frac{zf_{\alpha_{d}}''(z)}{f_{\alpha_{d}}'(z)}=1-(d+1)\sum_{n\geq1}\frac{z^{d+1}}{\gamma_{\alpha_{d},n}^{d+1}-z^{d+1}}.
\end{equation}
By making use of \eqref{HB15} and triangle inequality for $\left|z \right|<\gamma_{\alpha_{d},1},$ we get
\[\left|\left( 1+\frac{zf_{\alpha_{d}}''(z)}{f_{\alpha_{d}}'(z)}\right)^{2}-1 \right|\leq \left(\frac{\left| z\right| f_{\alpha_{d}}''(\left| z\right| )}{f_{\alpha_{d}}'(\left| z\right| )} \right)^{2}-2 \left(\frac{\left| z\right| f_{\alpha_{d}}''(\left| z\right| )}{f_{\alpha_{d}}'(\left| z\right| )} \right).\]
Hence, we deduce that the lemniscate convex radius of $f_{\alpha_{d}}$,  $r^{c}_{\mathcal{L}}(f_{\alpha_{d}}),$ is the unique positive root of the equation
\begin{equation}\label{HB16}
\left( \frac{rf_{\alpha_{d}}''(r)}{f_{\alpha_{d}}'(r)}\right)^{2}-2\left( \frac{rf_{\alpha_{d}}''(r)}{f_{\alpha_{d}}'(r)}\right)-1=0
\end{equation}
which implies 
\[\left( \frac{r^{2}\mathcal{J}_{\alpha_{d}}''(r)+2r\mathcal{J}_{\alpha_{d}}'(r)}{\mathcal{J}_{\alpha_{d}}(r)+r\mathcal{J}_{\alpha_{d}}'(r)}\right)^{2}-2\left(\frac{r^{2}\mathcal{J}_{\alpha_{d}}''(r)+2r\mathcal{J}_{\alpha_{d}}'(r)}{\mathcal{J}_{\alpha_{d}}(r)+r\mathcal{J}_{\alpha_{d}}'(r)} \right)-1=0. \]
Now, we need to show that  the above equation  has a unique root in $\left(0,\gamma_{\alpha_{d},1}\right).$ To do this, let us consider the function $v_{\alpha_{d}}:\left(0,\gamma_{\alpha_{d},1} \right)\rightarrow\mathbb{R} $ defined by
\[v_{\alpha_{d}}(z)=\left( \frac{rf_{\alpha_{d}}''(r)}{f_{\alpha_{d}}'(r)}\right)^{2}-2\left( \frac{rf_{\alpha_{d}}''(r)}{f_{\alpha_{d}}'(r)}\right)-1.\]
It is obvious that the function is continous and strictly increasing, since
\[v_{\alpha_{d}}'(r)>(d+1)^{3}\sum_{n\geq1}\frac{r^{d}\gamma_{\alpha_{d},n}^{d+1}}{\left(\gamma_{\alpha_{d},n}^{d+1}-r^{d+1} \right)^2 }\sum_{n\geq1}\frac{r^{d+1}}{\gamma_{\alpha_{d},n}^{d+1}-r^{d+1}}>0.\]
Observe also that
\[\lim_{r\searrow 0}v_{\alpha_{d}}(r)=-1<0 \text{ \ and \ } \lim_{r\nearrow \gamma_{\alpha_{d},1}}v_{\alpha_{d}}(r)=\infty.\]
Thus, we deduce that the root is unique in $\left(0,\gamma_{\alpha_{d},1} \right).$ This means that the lemniscate convex radius of $f_{\alpha_{d}}$ is the unique root of \eqref{HB16} in $\left(0,\gamma_{\alpha_{d},1} \right).$ 
\end{proof}

\subsection{Janowski starlikeness and Janowski convexity of normalized hyper-Bessel function}
In this section, we turn our attention to determining the radii of Janowski starlikeness and Janowski convexity of the normalized hyper-Bessel function $f_{\alpha_{d}}(z).$
\begin{theorem}
Let $\alpha_{i}>-1$ for $i\in \left\lbrace 1,2,\dots,d\right\rbrace .$ Then the Janowski starlikeness radius $r^{\star}_{A,B}(f_{\alpha_{d}})$ is the smallest positive root of the equation
\[\frac{r\mathcal{J}_{\alpha_{d}}'(r)}{\mathcal{J}_{\alpha_{d}}(r)}+\frac{A-B}{1+\left| B\right| }=0.\]
\end{theorem}
\begin{proof}
In order to determine the radius of Janowski starlikeness of the normalization $f_{\alpha_{d}}(z)$ of $\mathcal{J}_{\alpha_{d}}(z),$ we need to find a real number $r^{\star}$ such that
\[\left|\frac{\frac{zf_{\alpha_{d}}'(z)}{f_{\alpha_{d}}(z)}-1 }{A-B\frac{zf_{\alpha_{d}}'(z)}{f_{\alpha_{d}}(z)}} \right|<1  \quad (\left| z\right|<r^{\star} ).\]
In light of Eq. \eqref{HB11} and by using triangle inequality it follows that the inequality
\[\left|\frac{\frac{zf_{\alpha_{d}}'(z)}{f_{\alpha_{d}}(z)}-1 }{A-B\frac{zf_{\alpha_{d}}'(z)}{f_{\alpha_{d}}(z)}} \right|\leq\frac{(d+1)\sum\limits_{n\geq1}\frac{\left| z\right| ^{d+1}}{j_{\alpha_{d},n}^{d+1}-\left| z\right| ^{d+1}}}{A-B-\left| B\right|(d+1) \sum\limits_{n\geq1}\frac{\left| z\right| ^{d+1}}{j_{\alpha_{d},n}^{d+1}-\left| z\right| ^{d+1}}} \quad (\left| z\right| <j_{\alpha_{d},1})\]
holds for $\alpha_{i}>-1,$  $i\in \left\lbrace 1,2,\dots,d\right\rbrace $ with the equality at $z=\left| z\right| =r.$ With  the aid of Eq. \eqref{HB11} the above inequality yields
\begin{equation}\label{HB17}
\left|\frac{\frac{zf_{\alpha_{d}}'(z)}{f_{\alpha_{d}}(z)}-1 }{A-B\frac{zf_{\alpha_{d}}'(z)}{f_{\alpha_{d}}(z)}} \right|\leq \frac{1-\frac{\left| z\right| f_{\alpha_{d}}'(\left| z\right| )}{f_{\alpha_{d}}(\left| z\right| )}}{A-B+\left| B\right| \left( \frac{\left| z\right| f_{\alpha_{d}}'(\left| z\right| )}{f_{\alpha_{d}}(\left| z\right| )}-1\right) }.
\end{equation}
Hence we deduce that for $\left| z\right| <r^{\star},$ Janowski starlike radius   $r^{\star}_{A,B}(f_{\alpha_{d}})$ is the unique positive root of the equation
\[\frac{1-\frac{r f_{\alpha_{d}}'(r)}{f_{\alpha_{d}}(r)}}{A-B+\left| B\right| \left( \frac{r f_{\alpha_{d}}'(r)}{f_{\alpha_{d}}(r)}-1\right) }-1=0,\]
which implies
\begin{equation}\label{HB18}
\frac{rf'_{\alpha_{d}(r)}}{f_{\alpha_{d}(r)}}=1-\frac{A-B}{1+\left| B\right| }.
\end{equation}
 We need to show that  the above equation (that is Eq. \eqref{HB18})  has a unique root in $\left(0,j_{\alpha_{d},1}\right).$ Let us consider the function $u_{\alpha_{d}}:\left(0,j_{\alpha_{d},1} \right)\rightarrow\mathbb{R}$ defined by
 \[u_{\alpha_{d}}(r)=\frac{rf'_{\alpha_{d}(r)}}{f_{\alpha_{d}(r)}}-1+\frac{A-B}{1+\left| B\right| }.\]
 It is clear that the above mentioned function is continuous and strictly decreasing, since
 \[u_{\alpha_{d}}'(r)=-(d+1)^2\sum_{n\geq1}\frac{r^{d}j_{\alpha_{d},n}^{d+1}}{\left(j_{\alpha_{d},n}^{d+1}-r^{d+1} \right)^2 }<0.\]
 Observe also that
 \[\lim_{r\searrow 0}u_{\alpha_{d}}(r)=\frac{A-B}{1+\left| B\right| }>0 \text{ \ and \ } \lim_{r\nearrow j_{\alpha_{d},1}}u_{\alpha_{d}}(r)=-\infty.\]
 Therefore, the Intermediate Value Theorem ensures the existence of the unique root of $u_{\alpha_{d}}(r)=0$ in $\left( 0,j_{\alpha_{d},1}\right).$ That is, the Janowski starlikeness radius $r^{\star}_{A,B}(f_{\alpha_{d}})$ is the unique root of equation \eqref{HB18} in $\left( 0,j_{\alpha_{d},1}\right).$
\end{proof}
\begin{theorem}
Let $\alpha_{i}>-1$ for $i\in \left\lbrace 1,2,\dots,d\right\rbrace .$ Then the Janowski convexity radius $r^{c}_{A,B}(f_{\alpha_{d}})$ is the smallest positive root of the equation
\[\frac{r^{2}\mathcal{J}_{\alpha_{d}}''(r)+2r\mathcal{J}_{\alpha_{d}}'(r)}{\mathcal{J}_{\alpha_{d}}(r)+r\mathcal{J}_{\alpha_{d}}'(r)}+\frac{A-B}{1+\left| B\right| }=0.\]
\end{theorem}

\begin{proof}
In order that the function $f_{\alpha_{d}}$ is Janowski convex in the disk $\left\lbrace z: \left| z\right| <r \right\rbrace,$ the inequality
\[\left|\frac{\frac{zf_{\alpha_{d}}''(z)}{f_{\alpha_{d}}'(z)} }{A-B\left( 1+\frac{zf_{\alpha_{d}}''(z)}{f_{\alpha_{d}}'(z)}\right) }\right| <1\]
must be valid for $\left| z\right| <r.$
It is easily seen that the function $f_{\alpha_{d}}$ satisfies the inequality
\[\left|\frac{\frac{zf_{\alpha_{d}}''(z)}{f_{\alpha_{d}}'(z)} }{A-B\left( 1+\frac{zf_{\alpha_{d}}''(z)}{f_{\alpha_{d}}'(z)}\right) }\right|\leq \frac{(d+1)\sum\limits_{n\geq1}\frac{z^{d+1}}{\gamma_{\alpha_{d},n}^{d+1}-z^{d+1}}}{A-B-\left| B\right| \left((d+1)\sum\limits_{n\geq1}\frac{z^{d+1}}{\gamma_{\alpha_{d},n}^{d+1}-z^{d+1}}\right)}\]
for $\left| z\right| <\gamma_{\alpha_{d},1}$ with the equality at $z=\left| z\right|=r.$ The above inequality implies that
\begin{equation}\label{HB19}
\left|\frac{\frac{zf_{\alpha_{d}}''(z)}{f_{\alpha_{d}}'(z)} }{A-B\left( 1+\frac{zf_{\alpha_{d}}''(z)}{f_{\alpha_{d}}'(z)}\right) }\right|\leq \frac{-\frac{\left| z\right| f_{\alpha_{d}}''(\left| z\right| )}{f_{\alpha_{d}}'(\left| z\right| )} }{A-B+\left| B\right|\frac{\left| z\right| f_{\alpha_{d}}''(\left| z\right| )}{f_{\alpha_{d}}'(\left| z\right| )} }.
\end{equation}
In this case, we say that the Janowski convexity radius $r^{c}_{A,B}(f_{\alpha_{d}})$ is the smallest positive root of the equation
\begin{equation}\label{HB20}
\frac{rf_{\alpha_{d}}''(r)}{f_{\alpha_{d}}'(r)}+\frac{A-B}{1+\left| B\right| }=0.
\end{equation}
In order to finish the proof, we must show that the above mentioned equation (that is \eqref{HB20}) has a unique root in $\left(0,\gamma_{\alpha_{d},1} \right). $ To reach our aim, we consider the function $v_{\alpha_{d}}:\left(0,\gamma_{\alpha_{d},1} \right)\rightarrow \mathbb{R}$ defined by
\[v_{\alpha_{d}}(r)=\frac{rf_{\alpha_{d}}''(r)}{f_{\alpha_{d}}'(r)}+\frac{A-B}{1+\left| B\right| }.\]
It is obvious that the function $v_{\alpha_{d}}$ is strictly decreasing as
\[v_{\alpha_{d}}'(r)=-(d+1)^{2}\sum_{n\geq1}\frac{r^{d}\gamma_{\alpha_{d},n}^{d+1}}{\left(\gamma_{\alpha_{d},n}^{d+1}-r^{d+1} \right)^2 }<0.\]
Observe also that
\[\lim_{r\searrow 0}v_{\alpha_{d}}(r)=\frac{A-B}{1+\left| B\right| }>0 \text{ \ and \ } \lim_{r\nearrow \gamma_{\alpha_{d},1}}v_{\alpha_{d}}(r)=-\infty.\]
Therefore, by monotonicity of the function $v_{\alpha_{d}},$ it is obvious that the function $f_{\alpha_{d}}$ is Janowski convex for $\left| z\right| <r_{1},$ where $r_{1}$ is the unique positive root of equation \eqref{HB20} in $\left(0,\gamma_{\alpha_{d},1} \right).$
\end{proof}

\end{document}